\documentclass[12pt]{article}
\usepackage[toc,page]{appendix}
\usepackage[english]{babel}
\usepackage{amsmath}
\usepackage{amssymb,amsthm,amscd}
 \usepackage[pdftex]{color}
\usepackage[latin1]{inputenc}
\usepackage[textwidth=18cm,textheight=21cm]{geometry}
\usepackage{mathrsfs}
\usepackage[pdftex]{hyperref}

\newcommand{\diag}{\operatorname{diag}}

\newcommand{\C}{\mathbb{C}}

\newcommand{\R}{\mathbb{R}}

\newcommand{\ft}{\mathfrak t}
\newcommand{\fT}{\mathbb T}
\newcommand{\fK}{\mathbb K}
\newcommand{\fL}{\mathbb L}

\newcommand{\fk}{\mathfrak{k}}

\newcommand{\fl}{\mathfrak{l}}

\newcommand{\I}{\mathbb{I}}
\renewcommand{\d}{\mathrm{d}}

\newcommand{\ii}{\operatorname{i}}

\newtheorem{pro}{Proposition}

\newtheorem{definition}{Definition}
\newtheorem{theorem}{Theorem}
\newtheorem{cor}{Corollary}

\begin{document}


\title{Riemann--Hilbert problems, matrix orthogonal polynomials\\ and discrete matrix equations with singularity confinement}

\author{Giovanni A. Cassatella-Contra$^\dag$  and Manuel Ma\~{n}as$^\ddag$\\ Departamento de F\'{\i}sica Te\'{o}rica II\\(M\'{e}todos Matem\'{a}ticos de la F\'{\i}sica)\\Universidad Complutense de Madrid\\28040-Madrid, Spain}
\date{$^\dag$gaccontra@fis.ucm.es,  $^\ddag$manuel.manas@fis.ucm.es}
\maketitle
\begin{abstract}
In this paper matrix orthogonal polynomials in the real line are described in
terms of a Riemann--Hilbert problem. This approach provides an easy derivation
of discrete equations for the corresponding matrix recursion coefficients. The
discrete equation is explicitly derived in the matrix Freud case, associated
with matrix quartic potentials.  It is shown that, when the initial condition
and the measure are simultaneously triangularizable, this matrix discrete
equation possesses the singularity confinement property, independently if the
solution under consideration is given by recursion coefficients to quartic
Freud matrix orthogonal polynomials or not.
\end{abstract}

\maketitle
\section{Introduction}

The study of singularities of the solutions of nonlinear ordinary differential equations and, in particular, the quest of equations whose solutions are free of
movable critical points, the so called Painlev\'{e} property, lead, more than 110 years ago, to the Painlev\'{e} transcendents, see \cite{painleve} (and \cite{conte} for a recent account of the state of the art in this subject). The Painlev\'{e} equations are relevant in a diversity of fields, not only in Mathematics but also, for example in Theoretical Physics and in particular in 2D Quantum Gravity an Topological Field Theory, see for example \cite{conte}.

A discrete version of the Painlev\'{e} property, the singularity confinement
property, was introduced for the first time by Grammaticos, Ramani and
Papageorgiou in 1991 \cite{ramani2}, when they studied some discrete
equations, including the dPI equation (discrete version of the first
Painlev\'e equation), see also the contribution of these authors to
\cite{conte}. For this equation they realized that if eventually a singularity
could appear at some specific value of the discrete independent variable
it would disappear after performing few steps or iterations in the
equation. This property, as mentioned previously, is considered by these
authors as the equivalent of the Painlev\'e property \cite{painleve} for
discrete equations. Hietarinta and Grammaticos also derived some discrete
versions of the other five Painlev\'e equations \cite{ramani, ramani3}. See
also the interesting papers \cite{hietarinta} and \cite{lafortune}.

Freud orthogonal polynomials in the real line \cite{freud}
are associated to the weight  \begin{align*}
  w_{\rho}(x)&=|x|^{\rho}e^{-|x|^{m}},& \rho&>-1,\quad m>0.
\end{align*}
Interestingly, for $m=2, 4, 6 $ it has been shown \cite{assche} that from the recursion relation
\begin{align*}
xp_{n}=a_{n+1}p_{n+1}(x)+b_{n}p_{n}(x)+a_{n}p_{n-1}(x),
\end{align*}
the orthogonality of the polynomials leads to a recursion relation
    satisfied by the recursion coefficients $a_{n}$. In particular, for
$m=4$ Van Assche obtains for $a_{n}$ the discrete Painlev\'e I equations, and
therefore its singularities are confined. For related results see also
\cite{magnus}. For a modern and comprehensive account of this subject see the
survey \cite{simon}.

In 1992 it was found \cite{fokas} that the solution of a $2\times 2$
Riemann-Hilbert problem can be expressed in terms of orthogonal polynomials in
the real line and its Cauchy transforms. Later on this property has been used
in the study of certain properties of asymptotic analysis of orthogonal
polynomials and extended to other contexts, for example for the multiple
orthogonal polynomials of mixed type \cite{kuijlaars}.

 Orthogonal polynomials with matrix coefficients on the real line
have been considered in detail first by Krein \cite{krein1, krein2} in 1949,
and then were studied sporadically until the last decade of the XX
century. These are some papers of this subject: Berezanskii (1968)
\cite{bere}, Geronimo \cite{geronimo} (1982), and Aptekarev and Nikishin
\cite{nikishin} (1984). In the last paper they solved the scattering problem
for a kind of discrete Sturm--Liouville operators that are equivalent to the
recurrence equation for scalar orthogonal polynomials. They found that
polynomials that satisfy a recurrence relation of the form
\begin{align*}
  xP_{k}(x)&=A_{k}P_{k+1}(x)+B_{k}P_{k}(x)+A_{k-1}^{*}P_{k-1}(x),& k&=0,1,....
\end{align*}
 are orthogonal with respect to a positive definite measure. This is a
 matricial version of Favard's theorem for scalar orthogonal polynomials.
 Then, in the 1990's and the 2000's some authors found that matrix orthogonal
 polynomials (MOP) satisfy in certain cases some properties that satisfy
 scalar valued orthogonal polynomials; for example, Laguerre, Hermite and
 Jacobi polynomials, i.e., the scalar-type Rodrigues' formula
 \cite{duran20051,duran20052,constin} and a second order differential equation
 \cite{duran1997,duran2004,borrego}.

Later on, it has been proven \cite{duran2008} that operators of the form
$D$=${\partial}^{2}F_{2}(t)$+${\partial}^{1}F_{1}(t)$+${\partial}^{0}F_{0}$
have as eigenfunctions different infinite families of MOP's. Moreover, in
\cite{borrego} a new family of MOP's satisfying second order differential
equations whose coefficients do not behave asymptotically as the identity
matrix was found.  See also \cite{cantero}.

The aim of this paper is to explore the singularity confinement property in
the realm of matrix orthogonal polynomials. For that aim following
\cite{fokas} we formulate the matrix Riemann--Hilbert problem associated with
the MOP's. From the Riemann--Hilbert problem it follows not only the recursion
relations but also, for a type of matrix Freud weight with $m=4$, a nonlinear
recursion relation \eqref{eq:painleveS} for the matrix recursion coefficients,
that might be considered a matrix version --non Abelian-- of the discrete
Painlev\'{e} I. Finally, we prove that this matrix equation possesses the
singularity confinement property, and that after a maximum of 4 steps the
singularity disappears. This happens when the quartic potential $V$ and the
initial recursion coefficient are simultaneously triangularizable. It is
important to notice that the recursion coefficients for the matrix orthogonal
Freud polynomials provide solutions to \eqref{eq:painleveS} and therefore the
singularities are confined. A relevant fact for this solution is that the
collection of all recursion coefficients is an Abelian set of
matrices. However, not all solutions of \eqref{eq:painleveS} define a
commutative set; nevertheless, the singularity confinement still holds. In
this respect we must stress that our singularity confinement proof do not rely
in matrix orthogonal polynomials theory but only on the analysis of the
discrete equation. This special feature is not present in the scalar case
previously studied elsewhere.

The layout of this paper is as follows. In section 2 the Riemann--Hilbert problem for matrix orthogonal polynomials is derived and some of its consequences studied. In \S 3 a discrete matrix equation, for which the recursion coefficients of the Freud MOP's are solutions,  is derived and it is also proven that its singularities are confined. Therefore, it might be considered as a matrix discrete Painlev\'{e} I equation.

\section{Riemann--Hilbert problems and matrix orthogonal polynomials in the real line}
\subsection{Preliminaries on monic matrix orthogonal polynomials in the real line}

A family of matrix orthogonal polynomials (MOP's) in the real line
\cite{simon} is associated with a matrix-valued measure $\mu$ on $\R$; i. e.,
an assignment of a positive semi-definite $N\times N$ Hermitian matrix
$\mu(X)$ to every Borel set $X\subset\R$ which is countably additive. However,
in this paper we constraint ourself to the following case: given an $N\times
N$ Hermitian matrix $V(x)=(V_{i,j}(x))$, we choose $\d\mu=\rho (x)\d x$, being
$\d x$ the Lebesgue measure in $\R$, and with the weight function specified by
$\rho=\exp(-V(x))$ (thus $\rho$ is a positive semi-definite Hermitian
matrix). Moreover, we will consider only even functions in $x$, $V(x)=V(-x)$;
in this situation the finiteness of the measure $\d\mu$ is achieved for any
set of polynomials $V_{i,j}(x)$ in $x^2$. Associated with this measure we have
a unique family $\{P_n(x)\}_{n=0}^\infty$ of monic matrix orthogonal
polynomials
\begin{align*}
P_{n}(z)={\mathbb I}_{N}z^{n}+\gamma_{n}^{(1)}z^{n-1}+\cdots
+\gamma_{n}^{(n)}\in
{\mathbb C}^{N\times{N}},
\end{align*}
such that
\begin{align}
\int_{\mathbb R}P_{n}(x)x^{j}\rho(x)\text{d}x&=0,&j&=0,\dots,n-1.
\label{eq:ortogonal}
\end{align}
Here $\I_N$ denotes the identity matrix in $\C^{N\times N}$.

In terms of the moments of the measure $\d\mu$,
\begin{align*}
  m_{j}&:=\int_{\mathbb R}x^{j}\rho(x)
\text{d}x\in{\mathbb C}^{N{\times}N},&j&=0,1,\dots
\end{align*}
we define the truncated moment matrix
\begin{align*}
  m^{(n)}:=(m_{i,j})\in\C^{nN\times nN},
\end{align*}
with $m_{i,j}$=$m_{i+j}$ and $0\leq{i,j}\leq{n-1}$. Invertibility of $m^{(n)}$, i.e. $\det m^{(n)}\neq 0$, is equivalent to the existence of a unique family of monic matrix orthogonal polynomials.
In fact, we can write \eqref{eq:ortogonal} as
\begin{align}
\begin{pmatrix}
 m_{0}  &        \cdots       &   m_{n-1}   \\
   \vdots    &                  &    \vdots    \\
    m_{n-1}    &        \cdots      &   m_{2n-2}
\end{pmatrix}
\begin{pmatrix}
\gamma_{n}^{(n)}        \\
\vdots             \\
   \gamma_{n}^{(1)}
\end{pmatrix}
=\begin{pmatrix}
 -m_{n}     \\
 \vdots       \\
 -m_{2n-1}
\end{pmatrix},
\label{eq:zzz}
\end{align}
and hence uniqueness is equivalent to $\det m^{(n)}\neq 0$. From the
uniqueness and evenness we deduce that
\begin{align}%
P_{n}(z)={\mathbb I}_{N}z^{n}+\gamma_{n}^{(2)}z^{n-2}+\gamma_{n}^{(4)}z^{n-4}+\cdots
+\gamma_{n}^{(n)},
\label{eq:cinco}
\end{align}
where $\gamma_{n}^{(n)}=0$ if  $n$ is   odd.

The Cauchy transform of  $P_{n}(z)$ is defined by
\begin{align}
Q_{n}(z):=
\frac{1}{2{\pi}\ii}\int_{\mathbb R}\frac{P_{n}(x)}{x-z}\rho(x)\text{d}x,
\label{eq:q}
\end{align}
which is analytic for $z\in\C\backslash\R$. Recalling  $\frac{1}{z-x}=\frac{1}{z}\sum_{j=0}^{n-1}\frac{x^j}{z^j}+\frac{1}{z}\frac{(\frac{x}{z})^n}{1-\frac{x}{z}}$ and
  \eqref{eq:ortogonal}  we get
\begin{align}
Q_{n}(z)=-\frac{1}{2\pi\ii}\frac{1}{z^{n+1}}\int_{\mathbb R}
\frac{P_{n}(x)x^{n}}{1-\frac{x}{z}}\rho(x)\text{d}x,
\label{eq:qqu}
\end{align}
and consequently
\begin{align}
Q_{n}(z)&=c_{n}^{-1}z^{-n-1}+O(z^{-n-2}),
& z&\rightarrow\infty,
\label{eq:Qn}
\end{align}
where we have introduced the coefficients
\begin{align}
c_{n}:=\Big(-\frac{1}{2\pi\ii}
\int_{\mathbb R}P_{n}(x)\rho(x)x^{n}\d x\Big)^{-1},
\label{eq:defcn}
\end{align}
relevant in the  sequel of the paper.
\begin{pro}
  We have that $c_n$ satisfies
\begin{align}
{\det}c_{n}=-2\pi\ii\frac{{\det}(m^{(n)})}{{\det}(m^{(n+1)})}.
\label{eq:Cn}
\end{align}
\end{pro}
\begin{proof}
To prove it just define $\boldsymbol{m}:=(m_{n},...,m_{2n-1})$, consider the identity
\begin{align*}
 \begin{pmatrix}
{m^{(n)}}^{-1} & 0\\
0 & {\mathbb I}_{N}
\end{pmatrix}m^{(n+1)}=
\begin{pmatrix}
{\mathbb I}_{nN} & {m^{(n)}}^{-1}\boldsymbol{m}\\
\boldsymbol{m}^t & m_{2n}
\end{pmatrix},
\end{align*}
 and apply the Gauss elimination method to get
\begin{align*}
\frac{{\det}(m^{(n+1)})}{{\det}(m^{(n)})}=\det(m_{2n}
-{\boldsymbol{m}^t}{m^{(n)}}^{-1}{\boldsymbol{m}})\neq0;
\end{align*}
from \eqref{eq:zzz} we conclude $m_{2n}-{\boldsymbol{m}^t}{m^{(n)}}^{-1}{\boldsymbol{m}}=
\int_{\mathbb R}P_{n}(x)x^{n}\rho(x)\text{d}x$.
\end{proof}

The evenness of $V$ leads to $Q_{n}(z)=(-1)^{n+1}Q_{n}(-z)$, so that
\begin{align}
Q_{n}(z)&=c_{n}^{-1}z^{-n-1}+\sum_{j=2}^{\infty}a_{n}^{(2j-1)}z^{-n-2j+1},&
z&\rightarrow\infty.
\label{eq:Qn}
\end{align}
In particular,
\begin{align}
Q_{0}(z)&=c_{0}^{-1}z^{-1}+c_{1}^{-1}z^{-3}+O(z^{-5}),&z&\rightarrow\infty.
\label{eq:Q0}
\end{align}
Finally, if we assume that $V_{i,j}$ are H\"{o}lder continuous we get the
Plemelj formulae
\begin{align}
\Big(Q_{n}(z)_{+}-Q_{n}(z)_{-}\Big)\Big|_{x{\in}\mathbb R}=P_{n}(x)\rho(x),
\label{eq:plemelij}
\end{align}
with $Q_{n}(x)_{+} =Q_{n}(z)|_{z=x+i0^{+}}$ and
$Q_{n}(x)_{-} =Q_{n}(z)|_{z=x+i0^{-}}$.

\subsection{Riemann--Hilbert problem}
\begin{definition}\label{RH}
  The  Riemann--Hilbert problem to consider here is the finding of a $2N\times 2N$ matrix function
 $Y_{n}(z)\in\C^{2N\times 2N}$ such that
\begin{enumerate}
\item  $Y_{n}(z)$ is analytic in $z\in\C\backslash\R$.
\item Asymptotically behaves as
\begin{align}
Y_{n}(z)&=({\mathbb I}_{2N}+O(z^{-1}))\begin{pmatrix}
 {\mathbb I}_{N}z^{n} & 0  \\
0  &   {\mathbb I}_{N}z^{-n}
\end{pmatrix},&z&\to\infty.
\label{eq:r}
\end{align}
\item On $\R$ we have the jump
\begin{align}
Y_{n}(x)_{+}=Y_{n}(x)_{-}\begin{pmatrix}
 {\mathbb I}_{N} & \rho(x)   \\[8pt]
  0  &   {\mathbb I}_{N}
\end{pmatrix}.
\label{eq:rr}
\end{align}
\end{enumerate}
\end{definition}

An easy extension of the connection among orthogonal polynomials in the real
line with a particular Riemann--Hilbert problem discovered in \cite{fokas} can
be proven in this matrix context.
\begin{pro}
  The unique solution to the Riemann--Hilbert problem specified in Definition
  \ref{RH} is given in terms of monic matrix orthogonal polynomials with respect to $\rho(x)\d x$ and its
  Cauchy transforms:

\begin{align}
Y_{n}(z)&=\begin{pmatrix}
 P_{n}(z) & Q_{n}(z)  \\[8pt]
c_{n-1}P_{n-1}(z)  &   c_{n-1}Q_{n-1}(z)
\end{pmatrix}, &  n&\geq1.
\label{eq:y}
\end{align}
\end{pro}
\begin{proof}
  In the first place let us show that $\Big(\begin{smallmatrix}
 P_{n}(z) & Q_{n}(z)  \\[8pt]
c_{n-1}P_{n-1}(z)  &   c_{n-1}Q_{n-1}(z)
\end{smallmatrix}\Big)$ does satisfy the three conditions requested by Definition \ref{RH}.
\begin{enumerate}
\item As the matrix orthogonal polynomials $P_n$ are analytic in $\C$ and its
  Cauchy transforms are analytic in $\C\backslash \R$, the proposed solution
  is analytic in $\C\backslash \R$.
\item Replacing the asymptotics of the matrix orthogonal polynomials and its Cauchy transforms we get
$\Big(\begin{smallmatrix}
 P_{n}(z) & Q_{n}(z)  \\
c_{n-1}P_{n-1}(z)  &   c_{n-1}Q_{n-1}(z)
\end{smallmatrix}\Big)\to
\Big(\begin{smallmatrix}
z^n +O(z^{n-1})& O(z^{-n-1})  \\
O(z^{n-1})  &   z^{-n}+O(z^{-n-1})
\end{smallmatrix}\Big)= ({\mathbb I}_{2N}+O(z^{-1}))\Big(\begin{smallmatrix}
 {\mathbb I}_{N}z^{n} & 0  \\
0  &   {\mathbb I}_{N}z^{-n}
\end{smallmatrix}\Big)$ when $z\to\infty$.
\item From \eqref{eq:plemelij} we get $(\begin{smallmatrix}
 P_{n}(x+\ii 0) & Q_{n}(x+\ii 0)  \\
c_{n-1}P_{n-1}(x+\ii 0)  &   c_{n-1}Q_{n-1}(x+\ii 0)
\end{smallmatrix}\Big)-(\begin{smallmatrix}
 P_{n}(x-\ii 0) & Q_{n}(x-\ii 0)  \\
c_{n-1}P_{n-1}(x-\ii 0)  &   c_{n-1}Q_{n-1}(x-\ii 0)
\end{smallmatrix}\Big)=(\begin{smallmatrix}
 0& P_{n}(x)\rho(x)  \\
0  &   c_{n-1}P_{n-1}(x)\rho (x)
\end{smallmatrix}\Big).$
\end{enumerate}
Then, a solution to the RH problem is $Y_n=(\begin{smallmatrix}
 P_{n}(z) & Q_{n}(z)  \\
c_{n-1}P_{n-1}(z)  &   c_{n-1}Q_{n-1}(z)
\end{smallmatrix}\Big)$. But the solution is unique, as we will show now.  Given any solution $Y_n$, its determinant $\det Y_{n}(z)$  is analytic in $\C\backslash\R$ and satisfies
\begin{align*}
{\det}Y_{n}(x)_{+} &={\det}\Bigg(Y_{n}(x)_{-}\begin{pmatrix} {\mathbb I}_{N} &
  \rho(x) \\ 0 & {\mathbb I}_{N}
\end{pmatrix}
\Bigg)={\det}Y_{n}(x)_{-}{\det}
\begin{pmatrix}
  {\mathbb I}_{N} & \rho(x)   \\
 0  &    {\mathbb I}_{N}
\end{pmatrix}\\
&={\det}Y_{n}(x)_{-}.
\end{align*}
Hence,   $\det Y_{n}(z)$ is analytic in  $\mathbb C$. Moreover, Definition \ref{RH} implies that
\begin{align*}
  \det Y_{n}(z)&=1+O(z^{-1}),& z&\to\infty,
\end{align*}
and Liouville theorem ensures that
\begin{align}
{\det}Y_{n}(z)&=1, &\forall z\in\C.
\label{eq:det}
\end{align}
From \eqref{eq:det} we conclude that $Y_n^{-1}$ is analytic in
$\C\backslash\R$. Given two solutions $Y_n$ and $\tilde Y_n$ of the RH problem
we consider the matrix $\tilde Y_n Y_n^{-1}$, and observe that from
property 3 of Definition \ref{RH} we have $(\tilde Y_n Y_n^{-1})_+=(\tilde Y_n
Y_n^{-1})_-$, and consequently $\tilde Y_n Y_n^{-1}$ is analytic in $\C$. From
Definition \ref{RH} we get $\tilde Y_n Y_n^{-1}\to \I_{2N}$ as $z\to\infty$,
and Liouville theorem implies that $\tilde Y_n Y_n^{-1}=\I_{2N}$; i.e.,
$\tilde Y_n =Y_n$ and the solution is unique.
\end{proof}
\begin{definition}
Given the matrix $Y_n$ we define
\begin{align}
 S_{n}(z):= Y_{n}(z)\begin{pmatrix}
 {\mathbb I}_{N}z^{-n} & 0   \\
0  &   {\mathbb I}_{N}z^{n}
\end{pmatrix}.
\label{eq:igriega}
\end{align}
\end{definition}
\begin{pro}
\begin{enumerate}
\item The matrix $S_n$ has unit determinant:
 \begin{align}
\det S_{n}(z)=1.
\label{eq:detese}
\end{align}
\item It has the special form
  \begin{align}
S_{n}(z)=\begin{pmatrix}
A_{n}(z^2)  &  z^{-1}B_{n}(z^2) \\[8pt]
z^{-1}C_{n}(z^2) & D_{n}(z^2)
\end{pmatrix}.
\label{eq:22}
\end{align}
\item The coefficients of $S_n$ admit the asymptotic expansions
 \begin{align}\label{eq:abcd}
\begin{aligned}
  A_{n}(z^2)&={\mathbb I}_{N}+S^{(2)}_{n,11}z^{-2}+O(z^{-4}),&
B_{n}(z^2)&=S^{(1)}_{n,12}+S^{(3)}_{n,12}z^{-2}+O(z^{-4}),\\
C_{n}(z^2)&=S^{(1)}_{n,21}+S^{(3)}_{n,21}z^{-2}+O(z^{-4}),&
D_{n}(z^2)&={\mathbb I}_{N}+S^{(2)}_{n,22}z^{-2}+O(z^{-4}),
\end{aligned}
\end{align}
for $z\to\infty$.
\end{enumerate}
\end{pro}
\begin{proof}
\begin{enumerate}
\item Is a consequence of  \eqref{eq:det} and  \eqref{eq:igriega}.
\item It follows from the parity of $P_n$ and $Q_n$.
  \item  \eqref{eq:r} implies the following asymptotic behaviour
\begin{align}
S_{n}(z)&={\mathbb I}_{2N}
+S_{n}^{(1)}z^{-1}+O(z^{-2}),& z&\to\infty,
\label{eq:sSs}
\end{align}
and \eqref{eq:22} gives
\begin{align*}
  S_{n}^{(2i)}&=
\begin{pmatrix}
S_{n,11}^{(2i)} & 0   \\
0  &  S_{n,22}^{(2i)}
\end{pmatrix}, & S_{n}^{(2i-1)}&=
\begin{pmatrix}
0 & S_{n,12}^{(2i-1)}   \\
S_{n,21}^{(2i-1)}  &  0
\end{pmatrix},
\end{align*}
and the result follows.
\end{enumerate}
\end{proof}

Observe that from  \eqref{eq:22} we get
\begin{align}
S_{n}^{-1}(z)=\begin{pmatrix}
 \tilde{A}_{n}(z^{2}) & z^{-1}\tilde{B}_{n}(z^{2})\\
 z^{-1}\tilde{C}_{n}(z^{2}) & \tilde{D}_{n}(z^{2})
\end{pmatrix},
\label{eq:23}
\end{align}
with the asymptotic expansions for $z\rightarrow\infty$
\begin{align*}
\tilde{A}_{n}(z^2)&=
{\mathbb I}_{N}+(S^{(1)}_{n,12}S^{(1)}_{n,21}-S^{(2)}_{n,11})z^{-2}+O(z^{-4}),\\
\tilde{B}_{n}(z^2)&=
-S^{(1)}_{n,12}-\big(S^{(3)}_{n,12}-S^{(2)}_{n,11}S^{(1)}_{n,12}+
S^{(1)}_{n,12}(S^{(1)}_{n,21}S^{(1)}_{n,12}-S^{(2)}_{n,22})\big)z^{-2}
+O(z^{-4}),\\
\tilde{C}_{n}(z^2)&=
-S^{(1)}_{n,21}+\big(-S^{(3)}_{n,21}+S^{(1)}_{n,21}S^{(2)}_{n,11}+
(S^{(2)}_{n,22}-S^{(1)}_{n,21}S^{(1)}_{n,12})S^{(1)}_{n,21}\big)z^{-2}
+O(z^{-4}),\\
\tilde{D}_{n}(z^2)&=
{\mathbb I}_{N}+(S^{(1)}_{n,21}S^{(1)}_{n,12}-S^{(2)}_{n,22})z^{-2}+O(z^{-4}).
\end{align*}
\subsubsection{Recursion relations}
We now introduce the necessary elements, within the Riemann--Hilbert problem approach, to derive the recursion relations and properties of the recursion coefficients in the context of matrix orthogonal polynomials.
\begin{definition}
We introduce  the matrix
\begin{align}
Z_{n}(z):=Y_{n}(z)\begin{pmatrix}
\rho(z) & 0  \\
0 &  {\mathbb I}_{N}
\end{pmatrix}=\begin{pmatrix}
 P_{n}(z)\rho(z) & Q_{n}(z)  \\[8pt]
c_{n-1}P_{n-1}(z)\rho(z)  &   c_{n-1}Q_{n-1}(z)
\end{pmatrix}.
\label{eq:Z}
\end{align}
\end{definition}
%
%
%
%
%
%
\begin{pro}
\begin{enumerate}
\item $Z_{n}(z)$ is analytic on $\C\backslash\R$,
\item for $z\to\infty$ it holds that
\begin{align}
Z_{n}(z)=({\mathbb I}_{2N}+O(z^{-1}))
\begin{pmatrix}
 z^{n}\rho(z)  & 0  \\
0  &   z^{-n}{\mathbb I}_{N}
\end{pmatrix},\label{eq:zeta}
\end{align}
\item over $\R$ it is satisfied
\begin{align}
Z_{n}(x)_{+}=Z_{n}(x)_{-}\begin{pmatrix}
 {\mathbb I}_{N}  & {\mathbb I}_{N}  \\
0  &  {\mathbb I}_{N}
\end{pmatrix}.
\label{eq:discontinuidad}
\end{align}
\end{enumerate}
\end{pro}

\begin{definition}
We introduce
\begin{align}
M_{n}(z) &:=\frac{\d{Z}_{n}(z)}{\d{z}}Z_{n}^{-1}(z),
\label{eq:derivada2}\\
R_{n}(z) &:=Z_{n+1}(z){Z_{n}}^{-1}(z)=Y_{n+1}(z){Y_{n}}^{-1}(z).
\label{eq:R}
\end{align}
\end{definition}

We can easily show that
\begin{pro}
  The matrices $M_n$ and $R_n$ satisfy
  \begin{align}\label{comp}
    M_{n+1}(z)R_{n}(z)=\frac{\d}{\d{z}}R_{n}(z)+R_{n}(z)M_{n}(z).
  \end{align}
\end{pro}
\begin{proof}
It follows from the compatibility condition
\begin{align*}
  T\frac{\text{d}Z_{n}(z)}{\text{d}z}=\frac{\text{d}}{\text{d}z}TZ_{n}(z),
  \end{align*}
where $TF_{n}:=F_{n+1}$.
\end{proof}
We can also show that
\begin{pro}
 For the functions $R_{n}(z)$ and $M_n(z)$ we have the expressions
\begin{align}
R_{n}(z)&=\begin{pmatrix}
 z{\mathbb I}_{N}   & -S_{n,12}^{(1)} \\
S_{n+1,21}^{(1)}   & 0
\end{pmatrix},
\label{eq:Rn}\\
M_{n}(z)&=\Bigg[\begin{pmatrix}
A_{n}(z^2)\frac{\d\rho(z)}{\d{z}}
\rho^{-1}(z)\tilde{A}_{n}(z^2) & A_{n}(z^{2})z^{-1}
\frac{\d\rho(z)}{\d{z}}\rho^{-1}(z)\tilde{B}_{n}(z^2)\\[8pt]
z^{-1}C_{n}(z^2)\frac{\d\rho(z)}{\d{z}}\rho^{-1}(z)
\tilde{A}_{n}(z^2)& z^{-2}C_{n}(z^2)
\frac{\d\rho(z)}{\d{z}}\rho^{-1}(z)\tilde{B}_{n}(z^2)
\end{pmatrix}\Bigg]_{+},\label{eq:MMatrix}
\end{align}
where $[{\cdot}]_+$ denotes the part in positive powers of $z$.
\end{pro}
\begin{proof}
The expression for $R_n$ is a consequence of the following reasoning:
\begin{enumerate}
\item  In the first  place notice that $R_{n}(z)$ is analytic for $z\in\C\backslash\mathbb R$.
\item Moreover, denoting
\begin{align}
{R_{n}}_{+}(x) &:={Y_{n+1}}_{+}(x)({Y_{n}}_{+}(x))^{-1},\label{eq:rr+}\\
{R_{n}}_{-}(x) &:={Y_{n+1}}_{-}(x)({Y_{n}}_{-}(x))^{-1},
\end{align}
and substituting  \eqref{eq:rr} in \eqref{eq:rr+} we get  ${R_{n}}_{+}(x)={R_{n}}_{-}(x)$ and therefore
$R_{n}(z)$ is analytic in $\mathbb C$.
\item Finally, if we substitute \eqref{eq:igriega} in \eqref{eq:R} we deduce that
\begin{align*}
R_{n}(z) &=Y_{n+1}(z){Y_{n}}^{-1}(z)\\&=
S_{n+1}(z)\begin{pmatrix}
 z{\mathbb I}_{N}  & 0  \\
0  &  z^{-1}{\mathbb I}_{N}
\end{pmatrix}S_{n}^{-1}(z)\\
&=\begin{pmatrix}
 z{\mathbb I}_{N}  & 0  \\
0  &  0
\end{pmatrix}+S_{n+1}^{(1)}\begin{pmatrix}
 {\mathbb I}_{N}  & 0  \\
0  &  z^{-1}{\mathbb I}_{N}
\end{pmatrix}-\begin{pmatrix}
 {\mathbb I}_{N}  & 0  \\
0  &  0
\end{pmatrix}S_{n}^{(1)}+O(z^{-1}), &z&\rightarrow\infty,
\end{align*}
and the analyticity of $R_{n}$ in $\C$ leads to the desired result.
\end{enumerate}
For the expression for $M_n$ we have the  argumentation
\begin{enumerate}

\item $M_{n}(z)$ is analytic for $z\in\C\backslash\R$.
\item Given
\begin{align}
{M_{n}}_{+}(x) &:=\frac{{\text{d}}{Z_{n}}_{+}(x)}{{\text{d}}z}
({Z_{n}}_{+}(x))^{-1},\label{eq:tumaca}\\
{M_{n}}_{-}(x) &:=\frac{{\text{d}}{Z_{n}}_{-}(x)}
{{\text{d}}z}({Z_{n}}_{-}(x))^{-1}.
\end{align}
Substituting \eqref{eq:discontinuidad} in \eqref{eq:tumaca} we get
\begin{align*}
{  M_{n}}_{+}(x)={M_{n}}_{-}(x),
\end{align*}
and therefore $M_{n}(z)$ is analytic over $\C$.
\item From \eqref{eq:igriega} and  \eqref{eq:Z} we see that $Z_{n}(z)$ is
\begin{align}
   Z_{n}(z)=S_{n}(z)
 \begin{pmatrix}
  z^{n}\rho(z) & 0\\
 0 & z^{-n}{\mathbb I}_{N}
 \end{pmatrix},
 \label{eq:zgriega}
\end{align}
so that
\begin{align}
\frac{\text{d}Z_{n}(z)}{\text{d}z}Z_{n}^{-1}(z)=
\frac{\text{d}S_{n}(z)}{\text{d}z}S_{n}(z)^{-1}+
S_{n}(z)K_{n}(z)S_{n}^{-1}(z),
\label{eq:RRRRRRRRRRRRRRRR}
\end{align}
where
\begin{align*}
  K_{n}(z):=
 \begin{pmatrix}
  {n}z^{-1}{\mathbb I}_{N}+\dfrac{\text{d}\rho(z)}{\text{d}z}\rho^{-1}(z) &
0\\[8pt]
0 & -{n}z^{-1}{\mathbb I}_{N}
 \end{pmatrix}.
\end{align*}
Finally, as $M_{n}(z)$ is analytic over $\C$, \eqref{eq:RRRRRRRRRRRRRRRR}
leads to
\begin{align}
M_{n}(z)=\frac{\text{d}Z_{n}(z)}{\text{d}z}Z_{n}^{-1}(z)=
\Bigg[S_{n}(z)\begin{pmatrix}
 \frac{\text{d}\rho(z)}{\text{d}z}\rho^{-1}(z) & 0\\
 0 & 0
 \end{pmatrix}
S_{n}^{-1}(z)\Bigg]_{+}.
\label{eq:Mmatrix}
\end{align}
\end{enumerate}
\end{proof}
Observe that the diagonal terms of $M_n$ are odd functions of $z$ while the
off diagonal are even functions of $z$.  Now we give a parametrization of the
first coefficients of $S$ in terms of $c_n$.
\begin{pro}
The following formulae hold true
  \begin{align*}
S^{(1)}_{n,12}&=c_{n}^{-1}, & S^{(1)}_{n,21}&=c_{n-1},\\
S^{(2)}_{n,11} &=-\sum_{i=1}^{n}c_{i}^{-1}c_{i-1}+c_{n}^{-1}c_{n-1},&
S^{(2)}_{n,22} &=\sum_{i=1}^{n}c_{i-1}c_{i}^{-1},\\
S^{(3)}_{n,21}&=-c_{n-1}\sum_{i=1}^{n-1}c_{i}^{-1}c_{i-1}+c_{n-2},&
S^{(3)}_{n,12}&=c_{n}^{-1}\sum_{i=1}^{n+1}c_{i-1}c_{i}^{-1}.
  \end{align*}
\end{pro}
\begin{proof}
 Equating the expressions for  $Y_{n}(z)$ provided by \eqref{eq:y} and
\eqref{eq:igriega} we get
\begin{align*}
 Y_{n}(z) &=\begin{pmatrix}
P_{n}(z) & Q_{n}(z)\\
c_{n-1}P_{n-1}(z) & c_{n-1}Q_{n-1}(z)
\end{pmatrix}\\&=\begin{pmatrix}
 z^{n}{\mathbb I}_{N} & 0 \\
0 & z^{-n}{\mathbb I}_{N}
\end{pmatrix}\big({\mathbb I}_{2N}
+S^{(1)}_{n}z^{-1}+S^{(2)}_{n}z^{-2}+S^{(3)}_{n}z^{-3}
+O(z^{-4})\big),&z&\rightarrow\infty.
\end{align*}
Expanding the r.h.s. we get
\begin{align}\label{eq:uno}
\begin{aligned}
 S^{(1)}_{n,21}&=c_{n-1},&S^{(1)}_{n,12}&=c_{n}^{-1},&\\
 S^{(2)}_{1,11}&=0,\\
 S^{(3)}_{1,21}&=S^{(3)}_{2,21}=0, &S^{(3)}_{n,21}&=c_{n-1}S^{(2)}_{n-1,11},&n\geq2,\\
S^{(3)}_{n,12}&=c_{n}^{-1}S^{(2)}_{n+1,22},
\end{aligned}
\end{align}
where we have used that
\begin{align}
S^{(2)}_{1,22}=c_{0}c_{1}^{-1},
\label{eq:compatible0}
\end{align}
which can be proved from \eqref{eq:Q0}. Introducing \eqref{eq:uno} into
\eqref{eq:Rn} we get
\begin{align}
R_{n}(z)=\begin{pmatrix}
 z{\mathbb I}_{N}  & -c_{n}^{-1} \\[8pt]
c_{n} & 0
\end{pmatrix},
\label{eq:RnRn}
\end{align}
and  \eqref{eq:igriega} and \eqref{eq:RnRn} lead to
\begin{align*}
  S_{n+1}(z)&=\begin{pmatrix}
 z{\mathbb I}_{N}  & -c_{n}^{-1} \\[8pt]
c_{n} & 0
\end{pmatrix}S_{n}(z)
\begin{pmatrix}
z^{-1}{\mathbb I}_{N} &  0\\[8pt]
0  & z{\mathbb I}_{N}
\end{pmatrix},
\end{align*}
so that
\begin{align}
S^{(2)}_{n+1,11}-S^{(2)}_{n,11} &=-c_{n}^{-1}c_{n-1},\label{eq:d11}\\
S^{(3)}_{n,12}-c_{n}^{-1}S^{(2)}_{n,22} &=c_{n+1}^{-1},
\label{eq:cuarenta}
\end{align}
where we have used \eqref{eq:uno}. From
\eqref{eq:uno} and  \eqref{eq:cuarenta} we get
\begin{align}
S^{(2)}_{n+1,22}-S^{(2)}_{n,22}=c_{n}c_{n+1}^{-1}.
\label{eq:d22}
\end{align}
Summing up in $n$ in \eqref{eq:d11} and \eqref{eq:d22} we deduce
\begin{align*}
\sum_{i=1}^{n-1}(S^{(2)}_{i+1,11}-S^{(2)}_{i,11})
&=-\sum_{i=1}^{n-1}c_{i}^{-1}c_{i-1},\\
\sum_{i=1}^{n-1}(S^{(2)}_{i+1,22}-S^{(2)}_{i,22})&=
\sum_{i=1}^{n-1}c_{i}c_{i+1}^{-1},
\end{align*}
which leads to
\begin{align}
S^{(2)}_{n,11} &=-\sum_{i=1}^{n}c_{i}^{-1}c_{i-1}+c_{n}^{-1}c_{n-1},
\label{eq:ddd11}\\
S^{(2)}_{n,22} &=\sum_{i=1}^{n}c_{i-1}c_{i}^{-1},
\label{eq:ddd22}
\end{align}
where we have used \eqref{eq:uno} and \eqref{eq:compatible0}. Finally
\eqref{eq:uno}, \eqref{eq:ddd11} and \eqref{eq:ddd22} give
\begin{align}
S^{(3)}_{n,21}=-c_{n-1}\sum_{i=1}^{n-1}c_{i}^{-1}c_{i-1}+c_{n-2},
\label{eq:t21}
\end{align}
valid for $ n\geq2$, and
\begin{align}
S^{(3)}_{n,12}=c_{n}^{-1}\sum_{i=1}^{n+1}c_{i-1}c_{i}^{-1}.
\label{eq:t12}
\end{align}
\end{proof}
Notice that \eqref{eq:uno} gives
\begin{align}
S^{(3)}_{1,21}=0.
\label{eq:t123}
\end{align}
\begin{pro}
Matrix orthogonal polynomials $P_n$ (and its Cauchy transforms $Q_n$) are subject to the following recursion relations
  \begin{align}
P_{n+1}(z)=zP_{n}(z)-\frac{1}{2}\beta_{n}P_{n-1}(z),
\label{eq:recurrencia}
\end{align}
with the recursion coefficients $\beta_n$ given by
\begin{align}
\beta_{n}&:=2c_{n}^{-1}c_{n-1},& n&\geq 1,&\beta_0:=0.
\label{eq:cbeta}
\end{align}
\end{pro}
\begin{proof}
  Observe that \eqref{eq:R} can be written as
\begin{align}
Y_{n+1}(z)=R_{n}(z)Y_{n}(z).
\label{eq:mortadelo}
\end{align}
Then, if we replace \eqref{eq:y} and \eqref{eq:RnRn} into \eqref{eq:mortadelo} we get the result.
\end{proof}

We now show some commutative properties of the polynomials and the recursion
coefficients.
\begin{pro}\label{commutativity}
  Let $f(z):\C\to\C^{N\times N}$ such that $[V(x),f(z)]=0$ $\forall (x,z)\in\R\times \C$, then
  \begin{align*}
    [c_{n},f(z)]=[\beta_n,f(z)]&=0,&n&\geq0,& \forall z&\in\C,\\
    [P_n(z'),f(z)]&=0, & n&\geq 0,&\forall z,z'&\in\C.
  \end{align*}
\end{pro}
\begin{proof}
Let us suppose that for  a given $m\geq0$ we have
\begin{align}
[P_{m}(x),f(z)]=
[P_{m-1}(x),f(z)]=0.
\label{eq:CM01}
\end{align}
Then, recalling  \eqref{eq:defcn} these expressions  give
\begin{align}
[c_{m},f(z)]=[c_{m-1},f(z)]=0,
\label{eq:CM03}
\end{align}
respectively. Therefore, using the recursion relations \eqref{eq:recurrencia}
and \eqref{eq:cbeta} we obtain
\begin{align*}
 [ P_{m+1}(x),f(z)]=x[P_{m}(x),f(z)]-[c_{m}^{-1}c_{m-1}P_{m-1}(x),f(z)]=0.
\end{align*}
This means that
\begin{align}
[c_{m+1},f(z)]=0.
\label{eq:COM1}
\end{align}
Hypothesis \eqref{eq:CM01} holds for $m$=1, consequently $[c_n,f(z)]=0$ for
$n=0,1,\dots$ and \eqref{eq:cbeta} implies $[\beta_n,f(z)]=0$. Finally, as the
coefficients of the matrix orthogonal $P_n(z)$ are polynomials in the
$\beta$'s we conclude that $[P_n(z'),f(z)]=0$ for all $z,z'\in\C$.
\end{proof}
\begin{cor}
Suppose that $[V(x), V(z)]=0$ for all $x\in\R$ and $z\in\C$, then
  \begin{align}\label{Pc}
    [P_n(z),P_m(z')]&=0,& \forall n,m&\geq 0,& z,z'&\in\C,\\
     [c_n,c_m]&=0,\label{cc}\\
    [\beta_n,\beta_m]&=0. \label{betac}
  \end{align}
\end{cor}
\begin{proof}
 Applying  Proposition \ref{commutativity} to $f=V$ we deduce that  $[P_n(z'),V(z)]=0$, so that it allows to use  again Proposition \ref{commutativity} but now with  $f=P_n$ and get the stated result. From \eqref{eq:defcn} and \eqref{Pc} we deduce \eqref{cc} and using
\eqref{eq:cbeta} we get \eqref{betac}.
\end{proof}

\section{A discrete matrix  equation, related to Freud matrix orthogonal polynomials, with singularity confinement}
 We will consider the particular case when
\begin{align}\label{Freud potential}
  V(z)&=\alpha z^{2}+\I_Nz^4,& \alpha=\alpha^\dag.
\end{align}
Observe that $[V(z),V(z')]=0$ for any pair of complex numbers $z,z'$. Hence, in this case the corresponding set of  matrix orthogonal polynomials $\{P_n\}_{n=0}^\infty$, that we refer as matrix Freud polynomials, is an Abelian set. Moreover, we have
\begin{align*}
  [c_n,c_m]=[\beta_n,\beta_m]=[c_n,\alpha]=[\beta_n,\alpha]&=0, & \forall n,m=0,1,\dots.
\end{align*}

In this situation we have
\begin{theorem}
  The recursion coefficients $\beta_n$ \eqref{eq:cbeta} for the Freud matrix orthogonal
  polynomials determined by \eqref{Freud potential} satisfy
  \begin{align}
\beta_{n+1}&=n\beta_{n}^{-1}-\beta_{n-1}-\beta_{n}-\alpha,& n=1,2,\dots
\label{eq:painleveS}
\end{align}
with $\beta_0:=0$.
\end{theorem}
\begin{proof}
We compute now the matrix $M_n$, for which we have
\begin{align}
M_{n}(z)=\Bigg[\begin{pmatrix}
-A_{n}(z^2)(2{\alpha}z+4z^{3}{\mathbb I}_{N})\tilde{A}_{n}(z^2) &
-A_{n}(z^{2})(2{\alpha}
+4z^{2}{\mathbb I}_{N})\tilde{B}_{n}(z^2)\\
-C_{n}(z^2)(2{\alpha}+4z^{2}{\mathbb I}_{N})\tilde{A}_{n}(z^2)&
-C_{n}(z^2)(2{\alpha}z^{-1}+
4z{\mathbb I}_{N})\tilde{B}_{n}(z^2)
\end{pmatrix}\Bigg]_{+},
\label{eq:MMatrix}
\end{align}
and is clear that
\begin{align}
M_{n}(z)=M_{n}^{(3)}z^{3}+M_{n}^{(2)}z^{2}+M_{n}^{(1)}z+M_{n}^{(0)},
\label{eq:eme}
\end{align}
with
\begin{align*}
M_{n}^{(3)} &=\begin{pmatrix}
 -4{\mathbb I}_{N} & 0\\
 0 & 0
 \end{pmatrix},
 M_{n}^{(2)} =\begin{pmatrix}
 0 & 4S^{(1)}_{n,12} \\
 -4S^{(1)}_{n,21}  & 0
 \end{pmatrix},
M_{n}^{(1)} =\begin{pmatrix}
 -2\alpha-4S^{(1)}_{n,12}S^{(1)}_{n,21} &  0 \\
  0 & 4S^{(1)}_{n,21}S^{(1)}_{n,12}
 \end{pmatrix},\\
M_{n}^{(0)} &=\begin{pmatrix}
 0   & 2{\alpha}S^{(1)}_{n,12}+4S^{(3)}_{n,12}
+4S^{(1)}_{n,12}(S^{(1)}_{n,21}S^{(1)}_{n,12}-S^{(2)}_{n,22})\\
-2S^{(1)}_{n,21}{\alpha}-4S^{(3)}_{n,21}+
4S^{(1)}_{n,21}S^{(2)}_{n,11}-4S^{(1)}_{n,21}S^{(1)}_{n,12}S^{(1)}_{n,21} & 0
\end{pmatrix}.
\end{align*}
Replacing \eqref{eq:uno}-\eqref{eq:t12} into \eqref{eq:eme}
we get
\begin{align}
M_{n}^{(3)}&=\begin{pmatrix}
 -4{\mathbb I}_{N} & 0\\
 0 & 0
 \end{pmatrix},\quad
 M_{n}^{(2)}=\begin{pmatrix}
 0 & 4c_{n}^{-1} \\
 -4c_{n-1}  & 0
 \end{pmatrix},\quad
 M_{n}^{(1)}=\begin{pmatrix}
 -2\alpha-4 c_{n} ^{-1}c_{n-1} &  0 \\
  0  &    4c_{n-1}c_{n} ^{-1}
 \end{pmatrix},
\label{eq:matrices0}\\
M_{1}^{(0)}&=\begin{pmatrix}
 0   &  4c_{2}^{-1}+ 4c_{1}^{-1}c_{0}c_{1}^{-1}+2{\alpha}c_{1}^{-1}\\[8pt]
  -4c_{0}c_{1}^{-1}c_{0}-2c_{0}\alpha    &  0
\end{pmatrix},
\label{eq:matriz}\\
M_{n}^{(0)}&=\begin{pmatrix}
 0   &  4c_{n+1}^{-1}+ 4c_{n}^{-1}c_{n-1}c_{n}^{-1}
+2{\alpha}c_{n}^{-1}\\
  -4c_{n-2}-4c_{n-1}c_{n}^{-1}c_{n-1}
-2c_{n-1}\alpha  &  0
\end{pmatrix}, & n\geq 2.
\label{eq:matriz0}
\end{align}

The compatibility condition \eqref{comp} together with \eqref{eq:RnRn},
\eqref{eq:eme}, \eqref{eq:matrices0}, \eqref{eq:matriz} and
\eqref{eq:matriz0} gives
\begin{align*}
4(c_{n+2}^{-1}c_{n}+c_{n+1}^{-1}c_{n}c_{n+1}^{-1}c_{n}
-c_{n}^{-1}c_{n-1}c_{n}^{-1}c_{n-1}-c_{n}^{-1}c_{n-2})+2{\alpha}c_{n+1}^{-1}c_{n}
-2c_{n}^{-1}c_{n-1}\alpha={\mathbb I}_{N},
\end{align*}
for $n\geq 2$ and
\begin{align*}
4(c_{3}^{-1}c_{1}+c_{2}^{-1}c_{1}c_{2}^{-1}c_{1}-c_{1}^{-1}c_{0}c_{1}^{-1}c_{0})
+2{\alpha}c_{2}^{-1}c_{1}-2c_{1}^{-1}c_{0}\alpha={\mathbb I}_{N},
\end{align*}
which can be written as
\begin{align}
\beta_{n+2}\beta_{n+1}+\beta_{n+1}^2-\beta_{n}^2-\beta_{n}\beta_{n-1}+
\alpha\beta_{n+1}-\beta_{n}\alpha={\mathbb I}_{N}
\label{eq:compatibile1}
\end{align}
for $n\geq 2$ and
\begin{align}
\beta_{3}\beta_{2}+\beta_{2}^2-\beta_{1}^2+
\alpha\beta_{2}-\beta_{1}\alpha={\mathbb I}_{N},
\label{eq:compatibile2}
\end{align}
 respectively.
Using the Abelian character of the set of $\beta$'s we arrive to
\begin{align}
\beta_{n+2}\beta_{n+1}+\beta_{n+1}^2-\beta_{n}^2-\beta_{n}\beta_{n-1}+
\alpha(\beta_{n+1}-\beta_{n})&={\mathbb I}_{N}, & n&=2,3,\dots,
\label{eq:compatible1}\\
\beta_{3}\beta_{2}+\beta_{2}^2-\beta_{1}^2+
\alpha(\beta_{2}-\beta_{1})&={\mathbb I}_{N}.
\label{eq:compatible2}
\end{align}
Summing up in \eqref{eq:compatible1} from  $i$=2 up to $i$=$n$  we obtain
\begin{align}
\sum_{i=2}^{n}[\beta_{i+2}\beta_{i+1}+\beta_{i+1}^2-\beta_{i}^2
-\beta_{i}\beta_{i-1}+
\alpha(\beta_{i+1}-\beta_{i})]
=\sum_{i=2}^{n}{\mathbb I}_{N},
\label{eq:ccompatible}
\end{align}
and consequently we conclude
\begin{align}
\beta_{n+2}\beta_{n+1}+\beta_{n+1}\beta_{n}+{\beta_{n+1}}^{2}+
\alpha\beta_{n+1}&=n{\mathbb I}_{N}+k,&n&\geq 1,
\label{eq:integral}
\end{align}
where
\begin{align}
  k:=\beta_{2}\beta_{1}+\beta_{3}\beta_{2}
+\beta_{2}^2+\alpha\beta_{2}-{\mathbb I}_{N}=\beta_{2}\beta_{1}+\beta_{1}^2+\beta_{1}\alpha,\label{eq:kappa}
\end{align}
where we have used \eqref{eq:compatible2}. We now proceed to show that
$k$=${\mathbb I}_{N}$. \eqref{eq:derivada2} implies, for $n$=1 and $z=0$,
\begin{align}
Z'_{1}(0)=M_{1}^{(0)}Z_{1}(0),
\label{eq:derivada0}
\end{align}
with $M_{1}^{(0)}$ given in \eqref{eq:matriz}. This leads to
\begin{align}
\begin{pmatrix}
P'_{1}(0)\\
c_{0}P'_{0}(0)
\end{pmatrix}
=M_{1}^{(0)}\begin{pmatrix}
P_{1}(0)\\
c_{0}P_{0}(0)
\end{pmatrix}.
\end{align}
Now, using \eqref{eq:cinco} we deduce that
\begin{align}
\begin{pmatrix}
{\mathbb I}_{N}\\
0
\end{pmatrix}
=M_{1}^{(0)}\begin{pmatrix}
   0\\
c_{0}
\end{pmatrix},
\end{align}
which allows us to immediately deduce that
\begin{align}
\beta_{2}\beta_{1}+\beta_{1}^2+\beta_{1}\alpha={\mathbb I}_{N},
\label{eq:condicioninicial000}
\end{align}
and consequently $k$=${\mathbb I}_{N}$.
Finally, we get
\begin{align}
\beta_{n+2}\beta_{n+1}+\beta_{n+1}\beta_{n}+{\beta_{n+1}}^{2}+
\alpha\beta_{n+1}=n{\mathbb I}_{N}+{\mathbb I}_{N}.
\label{eq:Integral}
\end{align}
Finally, notice that \eqref{eq:condicioninicial000} reads
\begin{align}
\beta_{2}=\beta_{1}^{-1}-\beta_{1}-\alpha.
\label{eq:condicioninicial}
\end{align}
\end{proof}
This theorem ensures that $\beta_1$ fixes $\beta_n$ for all $n\geq 2$, and therefore $\beta_n=\beta_n(\beta_1,\alpha)$.  Moreover, we will see now that the solutions $\beta_n$ not only commute with each other but also that they can be simultaneously
conjugated to lower matrices. This result is relevant in our analysis of the  confinement of singularities.

\subsection{On singularity confinement}

The study of the singularities of the discrete matrix equations
\eqref{eq:painleveS} reveals, as we will show, that they are confined;
i.e. the singularities may appear eventually, however they disappear in few
steps, no more than four. The mentioned singularities in \eqref{eq:painleveS}
appear when for some $n$ the matrix $\beta_n$ is not invertible, that is
$\det\beta_n=0$, and we can not continue with the sequence. However, for a
better understanding of this situation in the discrete case we just request
that $\det\beta_n$ is a small quantity so that $\beta_n^{-1}$ and
$\beta_{n+1}$ exist, but they are very ``large" matrices in some appropriate
sense. To be more precise we will consider a small parameter $\epsilon$ and
suppose that in a given step $m$ of the sequence we have
\begin{align}
\label{eq:Condicioninicial1}
 \beta_{m-1}&=O(1),&
\det\beta_{m-1}&=O(1),
\\
\label{eq:Condicioninicial2}
  \beta_{m}&=O(1),&
\det\beta_{m}&=O(\epsilon^r),
\end{align}
for $\epsilon\to 0$ and  with $r\leq N-1$. In other words, we have the asymptotic expansions
\begin{align}
\beta_{m-1}&=\beta_{m-1,0}+
\beta_{m-1,1}\epsilon+O(\epsilon^2),&\epsilon&\rightarrow0, & \det\beta_{m-1,0}&\neq 0,
\label{eq:betacero}\\
\beta_{m}&=\beta_{m,0}+\beta_{m,1}\epsilon+O(\epsilon^2),&\epsilon&\rightarrow0,& \dim\operatorname{Ran}\beta_{m,0}&=N-r.
\label{eq:betauno}
\end{align}

We now proceed with some preliminar material. In particular we show that we
can restrict the study to the triangular case.
\begin{pro}\label{trian}
Let us suppose that $\beta_1$ and $\alpha$ are simultaneously triangularizable matrices; i.e.,  there exist an invertible matrix $M$ such that $\beta_1=M\phi_1 M^{-1}$ and $\alpha =M\gamma M^{-1}$ with $\phi_1$ and $\gamma$ lower triangular matrices. Then, the solutions $\beta_n$ of \eqref{eq:painleveS} can be written as
  \begin{align*}
    \beta_n&=M\phi_nM^{-1},& n\geq 0,
  \end{align*}
 where $\phi_n$, $n=0,1,\dots$, are lower triangular matrices satisfying
  \begin{align*}
    \phi_{n+1}=n\phi_n^{-1}-\phi_{n-1}-\phi_n-\gamma.
  \end{align*}
  Moreover, let us suppose that for some integer $m$ the matrices
  $\beta_{m+1}$, $\beta_m$ and $\alpha$ are simultaneously triangularizable,
  then all the sequence $\{\beta_n\}_{n=0}^\infty$ is simultaneously
  triangularizable.
\end{pro}
\begin{proof}
  In the one hand, from \eqref{eq:painleveS} we conclude that $M^{-1}\beta_2
  M$ is lower triangular and in fact that $\{M^{-1}\beta_n M\}_{n\geq 0}$ is a
  sequence of lower triangular matrices. In the other hand, if for some
  integer $m$ the matrices $\beta_{m+1}$, $\beta_m$ and $\alpha$ are
  simultaneously triangularizable we have
    \begin{align*}
    \beta_{m+1}&=m\beta_m^{-1}-\beta_{m}-\beta_{m-1}-\alpha,\\
     \beta_{m}&=(m-1)\beta_{m-1}^{-1}-\beta_{m-1}-\beta_{m-2}-\alpha,
  \end{align*}
  which implies that $\beta_{m-1},\beta_{m-2}$ are triangularized by the same
  transformation that triangularizes $\beta_{m+1}$, $\beta_m$ and $\alpha$.
\end{proof}

The simultaneous triangularizability can be achieved, for example, when
$[\beta_1,\alpha]=0$, as in this case we can always find an invertible matrix
$M$ such that $\beta_1=M\phi_1M^{-1}$ and $\alpha=M \gamma M^{-1}$ where
$\phi_1$ and $\gamma$ are lower triangular matrices, for example by finding
the Jordan form of these two commuting matrices. This is precisely the
situation for the solutions related with matrix orthogonal
polynomials. Obviously, this is just a sufficient condition.  From now on, and
following Proposition \ref{trian}, we will assume that the simultaneous
triangularizability of $\alpha$ and $\beta_1$ holds and study the case in
where $\alpha$ and all the $\beta$'s are lower triangular matrices. Thus, we
will use the splitting
\begin{align}
\label{eq:betadn}\beta_{n}&=D_{n}+N_{n},\\
\label{eq:alfadn}\alpha&=\alpha_{D}+\alpha_{N},
\end{align}
where $D_{n}=\diag(D_{n;1},\dots, D_{n;N})$ and
$\alpha_{D}=\diag(\alpha_{D,1},\dots,\alpha_{D,N})$ are the diagonal parts of
$\beta_{n}$ and $\alpha$, respectively and $N_{n}$ and $\alpha_{N}$ are the
strictly lower parts of $\beta_{n}$ and $\alpha$, respectively. Then,
\eqref{eq:painleveS} splits into
\begin{align}\label{eq:d+n}
\begin{aligned}
  D_{n+1}+N_{n+1}= {n}D_{n}^{-1}-D_{n-1}-D_{n}-\alpha_{D}\\+
n\bar{N}_{n}-N_{n-1}-N_{n}-\alpha_{N},
\end{aligned}
\end{align}
where $\bar{N}_{n}$ denotes the strictly lower triangular part $\beta_{n}^{-1}$; i.e.,
\begin{align*}
  \beta_{n}^{-1}=D_{n}^{-1}+\bar{N}_{n}.
\end{align*}
Hence, \eqref{eq:painleveS} decouples into
\begin{align}
D_{n+1}= {n}D_{n}^{-1}-D_{n-1}-D_{n}-\alpha_{D},
\label{eq:painlevediag2}\\
N_{n+1}= n\bar{N}_{n}-N_{n-1}-N_{n}-\alpha_{N}.
\label{eq:painlevenil}
\end{align}
In this context it is easy to realize that there always exists a transformation leading to the situation in where
\begin{align}
\label{eq:betauno2}
\beta_{m,0}=\begin{pmatrix}
 0 & 0       & \cdots    & 0 & 0 &\cdots &0\\
0 & 0 & \cdots & 0 & 0 &\cdots&0\\
\vdots & \vdots & & \vdots & \vdots & &\vdots\\
\beta_{m,0;r+1,1} & \beta_{m,0;r+1,2} & \cdots& \beta_{m,0;r+1,r+1}& 0 & \cdots &
0\\
\beta_{m,0;r+2,1} & \beta_{m,0;r+2,2} & \cdots & \beta_{m,0;r+2,r+1} &
\beta_{m,0;r+2,r+2} & \cdots & 0\\ \vdots & \vdots & &
\vdots&\vdots &  &\vdots \\ \beta_{m,0;N,1} &
\beta_{m,0;N,2} & \cdots& \beta_{m,0;N,r+1} & \beta_{m,0;N,r+2} & \cdots
&\beta_{m,0;N,N} \\
\end{pmatrix}.
\end{align}
\begin{pro}\label{conf_diagonal}
  The singularities of the diagonal part are confined. More explicitly, if we assume that \eqref{eq:betacero},
\eqref{eq:betauno} and \eqref{eq:betauno2} hold true  at a given step $m$ then
\begin{align}
D_{m+1;i}&=\frac{m}{\beta_{m,1;i,i}}\epsilon^{-1}
-\beta_{m-1,0;i,i}-\frac{\beta_{m,2;i,i}m}{\beta_{m,1;i,i}^2}-\alpha_{D,i}
+O(\epsilon),\notag\\
D_{m+2;i}&=-\frac{m}{\beta_{m,1;i,i}}\epsilon^{-1}
+\beta_{m-1,0;i,i}+\frac{\beta_{m,2;i,i}m}{\beta_{m,1;i,i}^2}+O(\epsilon),\notag\\
D_{m+3;i}&=-\beta_{m,1;i,i}\frac{m+3}{m}\epsilon+O(\epsilon^2),
\label{eq:d3}\\
D_{m+4;i}&=\frac{m\beta_{m-1,0;i,i}-2\alpha_{D,i}}{m+3}+O(\epsilon),
\label{eq:diagconf}
\end{align}
when $\epsilon\to 0$.
\end{pro}
\begin{proof}
From \eqref{eq:betacero}, \eqref{eq:betauno} and \eqref{eq:betauno2} we deduce
\begin{align*}
 D_{m-1,i}&=\beta_{m-1,0;i,i}+\beta_{m-1,1;i,i}\epsilon+O(\epsilon^2),\\
D_{m,i}&=\beta_{m,1;i,i}\epsilon+O(\epsilon^2),
\end{align*}
for $\epsilon\to 0$, with $i\geq r+1$. Substitution of these expressions in
\eqref{eq:painlevediag2} leads to the stated formulae.  For $i\leq r$ the
coefficients $D_{m-1;i}$ and $D_{m;i}$ are $O(1)$ as $\epsilon\to 0$, thus
they do not vanish, and consequently there is confinement of singularities for
the diagonal part $D_n$.
\end{proof}

In what follows we will consider asymptotic expansions taking values in the set of lower triangular matrices
\begin{align}
  \fT&:=\{T_0+T_1\epsilon+O(\epsilon^2),\; \epsilon\to 0, \quad T_i\in\ft_N  \}, &\ft_N&:=\{T=(T_{i,j})\in\C^{N\times N},\quad X_{i,j}=0 \text{ when $i>j$}\},\label{eq:tipoT}
\end{align}
where $\ft_N$ is the set of lower triangular $N\times N$ matrices. The reader should notice that this set $\fT=\ft_N[[\epsilon]]$ is a subring of the ring of $\C^{N\times N}$-valued asymptotic expansions; in fact is a subring with identity, the matrix $\I_N$. We will use the notation
\begin{align}
T_i&:=\begin{pmatrix}
  T_{i,11} & 0   \\
 T_{i,21}  &  T_{i,22}
\end{pmatrix},& i&\geq 1,
\end{align}
where $T_{i,11}\in\ft_r$, $T_{i,22}\in\ft_{N-r}$ and
$T_{i,21}\in \C^{(N-r)\times r}$.
We consider two sets of matrices determined by \eqref{eq:betauno2}, namely
\begin{align*}
  \fk&:=\Big\{ K_0=\begin{pmatrix}
  0 & 0   \\
 K_{0,21}  &  K_{0,22}
\end{pmatrix},
K_{0,21}\in\C^{(N-r)\times r}, K_{0,22}\in\ft_{N-r}\Big\},\\
 \fl&:=\{ L_{-1}=\begin{pmatrix}
  L_{-1,11} & 0   \\
 L_{-1,21}  &  0
\end{pmatrix},\;   L_{-1,11}\in\ft_{r}, L_{-1,21}\in\C^{(N-r)\times r}\Big\},
 \end{align*}
and the related sets
\begin{align}
  \fK&:=\Big\{K=K_0+K_1\epsilon+O(\epsilon^2)\in\fT,\quad K_0\in\fk\Big\},
\label{eq:tipoK}\\
\fL&:=\Big\{L=L_{-1}\epsilon^{-1}+L_{0}+L_{1}\epsilon+O(\epsilon^2)\in\epsilon^{-1}\fT,\quad L_{-1}\in\fl\Big\},
\label{eq:tipoL}
\end{align}
which fulfill the following important properties.
\begin{pro}\label{pro_ring}
  \begin{enumerate}
    \item Both $\fK$ and $\epsilon\fL$ are subrings of the ring with identity
      $\fT$, however these two subrings have no identity.
    \item If an element $X\in\fK$ is an invertible matrix, then
      $X^{-1}\in\fL$, and reciprocally if $X\in\fL$ is invertible, then
      $X^{-1}\in\fK$.
    \item The subrings $\epsilon\fL$ and $\fK$ are bilateral ideals of $\fT$;
      i.e., $\fL\cdot\fT\subset\fL$, $\fT\cdot\fL\subset\fL$,
      $\fT\cdot\fK\subset\fK$ and $\fK\cdot\fT\subset\fK$.
    \item We have $\fL\cdot\fK\subset\fT$.
  \end{enumerate}
\end{pro}

\begin{theorem}
  If $\beta_1$ and $\alpha$ are simultaneously triangularizable matrices then the singularities of \eqref{eq:painleveS} are confined. More explicitly,  if for a  given step $m$ the conditions \eqref{eq:betacero},  \eqref{eq:betauno} and \eqref{eq:betauno2} are satisfied
  then
  \begin{align*}
    \beta_{m+1},\beta_{m+2}&\in\fL, &\beta_{m+3}&\in\fK, &\beta_{m+4}&\in\fT,
    &\det\beta_{m+4}&=O(1),\quad \epsilon\to 0.
  \end{align*}
\end{theorem}
\begin{proof}
From \eqref{eq:betauno} and \eqref{eq:betauno2} we conclude that
$\beta_m\in\fK$ and consequently $\beta_m^{-1}\in\fL$. Taking into account
this fact, \eqref{eq:painleveS} implies that $\beta_{m+1}\in\fL$. Therefore,
$\beta_{m+1}^{-1}\in\fK$ and \eqref{eq:painleveS}, as $\beta_{m+1}\in\fL$,
give $\beta_{m+2}\in\fL$ and consequently $\beta_{m+2}^{-1}\in\fK$.  Iterating
\eqref{eq:painleveS} we get
\begin{align}
\beta_{m+3}=\beta_{m}-(m+1)\beta_{m+1}^{-1}+(m+2)\beta_{m+2}^{-1}.
\label{eq:painleves}
\end{align}
Using the just derived facts, $\beta_{m+1}^{-1},\beta_{m+2}^{-1}\in\fK$, and
that $\beta_m\in\fK$, we deduce $\beta_{m+3}\in\fK$ which implies
$\beta_{m+3}^{-1}\in\fL$.  Finally, \eqref{eq:painleveS} gives $\beta_{m+4}$
as
\begin{align}
\beta_{m+4}=(m+3)\beta_{m+3}^{-1}-\beta_{m+2}-\beta_{m+3}-\alpha.
\label{eq:painleveS4}
\end{align}
 We conclude that there are only two possibilities:
\begin{enumerate}
  \item $\beta_{m+4}=O(1)$ for $\epsilon\to 0$, or
  \item $\beta_{m+4}\in\fL$.
\end{enumerate}
Let us consider both possibilities separately.
\begin{enumerate}
  \item Recalling that  the diagonal part has singularity confinement, see Proposition \ref{conf_diagonal},
in the first case we see that $\det\beta_{m+4}=O(1)$ when $\epsilon\to 0$, as desired.
\item In this second case we write $\beta_{m+4}$ as
\begin{align}
\beta_{m+4}&=\beta_{m+3}^{-1}A+O(1),&\epsilon&\rightarrow0,& A:=(m+3){\mathbb I}-\beta_{m+3}\beta_{m+2}.
\label{eq:confi}
\end{align}
Observe that repeated use of \eqref{eq:painleveS} leads to the following
expressions:
\begin{align*}
  A&={\mathbb I}+[(m+1)\beta_{m+1}^{-1}-\beta_{m}]\beta_{m+2}\\
  &=k+{\mathbb I}-[(m+1)\beta_{m+1}^{-1}
-\beta_{m}]\beta_{m+1}\\&=k-m{\mathbb I}+\beta_{m}\beta_{m+1}\\
&=k-\beta_{m}(\beta_{m}
+\beta_{m-1}+\alpha),
\end{align*}
where
\begin{align*}
  k:=[(m+1)\beta_{m+1}^{-1}
-\beta_{m}][(m+1)\beta_{m+1}^{-1}-\beta_{m}-\alpha].
\end{align*}
From these formulae, as $\beta_{m+1}^{-1},\beta_m\in\fK$ we deduce that $k\in \fK$ and also that $\beta_{m}(\beta_{m}
+\beta_{m-1}+\alpha)\in\fK$. Hence, we conclude that $A\in\fK$ and from \eqref{eq:confi} and Proposition \ref{pro_ring} we deduce that $\beta_{m+4}=O(1)$ when $\epsilon\to0$. Consequently, we arrive to a contradiction, and only possibility 1) remains.
\end{enumerate}

\end{proof}

\section*{Acknowledgements}

The authors thanks economical support from the Spanish Ministerio de Ciencia e
Innovaci\'{o}n, research project FIS2008-00200. GAC acknowledges the support
of the grant Universidad Complutense de Madrid. Finally, MM reckons
illuminating discussions with Dr. Mattia Cafasso in relation with
orthogonality and singularity confinement, and both authors are grateful to
Prof. Gabriel \'{A}lvarez Galindo for several discussions and for the
experimental confirmation, via \textsc{Mathematica}, of the existence of the
confinement of singularities in the $2\times 2$ case.

\end{document}